\documentclass[preprint,12pt]{elsarticle}

\usepackage{amsmath, amsthm, amssymb}
\usepackage{enumitem}

\newtheorem{definition}{Definition}
\newtheorem{theorem}{Theorem}

\newtheorem{lemma}{Lemma}

\newtheorem{remark}{Remark}

\renewcommand{\vec}[1]{\boldsymbol{\mathbf{#1}}}

\renewcommand{\ker}[1]{\mathrm{ker}(#1)}
\newcommand{\vspan}[1]{\mathrm{span}(\{#1\})}

\newcommand{\diag}[1]{\mathrm{diag}(#1)}

\newcommand{\iu}{{i\mkern1mu}}
\newcommand{\inner}[2]{\langle #1, #2 \rangle_H}

\begin{document}

\begin{frontmatter}
\title{On an eigenvalue property of Summation-By-Parts operators}
\author[LU]{Viktor Linders}
\ead{viktor.linders@math.lu.se}
\address[LU]{Center for Mathematical Sciences,
Lund University,
Lund,
Sweden}


\begin{abstract}
Summation-By-Parts (SBP) methods provide a systematic way of constructing provably stable numerical schemes. However, many proofs of convergence and accuracy rely on the assumption that the SBP operator possesses a particular eigenvalue property. In this note, three results pertaining to this property are proven. Firstly, the eigenvalue property does not hold for all nullspace consistent SBP operators. Secondly, this issue can be addressed without affecting the accuracy of the method by adding a specially designed, arbitrarily small perturbation term to the SBP operator. Thirdly, all pseudospectral methods satisfy the eigenvalue property.
\end{abstract}


\end{frontmatter}


\section{Introduction} \label{sec:Introduction}

The Summation-By-Parts (SBP) methodology \cite{fernandez2014review,svard2014review} and its extensions (generalized SBP \cite{fernandez2014generalized} and upwind SBP \cite{mattsson2017diagonal}) constitute an algebraic framework for designing provably stable discretizations of partial differential equations. SBP operators can be designed within essentially every family of spatial discretizations, including finite difference methods \cite{kreiss1974finite,strand1994summation}, finite volume methods \cite{nordstrom2001finite}, finite element methods \cite{abgrall2020analysis,abgrall2021analysis}, pseudospectral methods \cite{carpenter1996spectral} and the related discontinuous Galerkin \cite{gassner2013skew} and Flux Reconstruction methods \cite{ranocha2016summation}, as well as more specialized discretizations such as WENO \cite{yamaleev2009systematic,fisher2011boundary} and DRP schemes \cite{linders2016summation,linders2017summation}. Further, they can be used as time marching schemes akin to implicit Runge-Kutta methods \cite{nordstrom2013summation,lundquist2014sbp,boom2015high}, in which case they satisfy an assortment of desirable stability and convergence properties \cite{linders2020properties,nordstrom2018well}. All such methods share certain algebraic similarities that can be exploited to obtain stable and accurate discretizations.

This paper addresses a difficulty pertaining to an eigenvalue property that frequently arises in SBP theory. Much of what is known about the convergence of SBP methods relies on this property. It is known that there are methods that lack the property, essentially due to the existence of SBP operators with "bad" null-spaces. Here, we consider SBP operators that do not suffer from bad nullspaces; so called nullspace consistent methods.

Three new results will be presented: After introducing nullspace consistent SBP operators in Section~\ref{sec:SBP}, it is shown that not every such operator satisfies the eigenvalue property. In Section~\ref{sec:Perturbations} it is established that for each SBP operator that lacks the property, another SBP operator of the same order can be found that possesses it, and that differs from the first one by an arbitrarily small perturbation. Section~\ref{sec:Pseudospectral} discusses pseudospectral methods, all of which are shown to satisfy the eigenvalue property. This generalizes an earlier result on the topic \cite{ruggiu2018pseudo}. A summary is given in Section~\ref{sec:Summary}.


\section{SBP operators and the eigenvalue property} \label{sec:SBP}

This section introduces the notion of nullspace consistent SBP operators as well as the eigenvalue property. Let $[a,b]$ be an interval with $b>a$.

\begin{definition} \label{def:SBP}
The matrices $D_+$, $D_- \in \mathbb{R}^{(n+1) \times (n+1)}$ are said to form a pair of SBP operators of order $q \geq 1$ on the interval $[a,b]$ if there exist matrices $H$, $S \in \mathbb{R}^{(n+1) \times (n+1)}$ and vectors $\vec{p}_0$, $\vec{p}_n$, $\vec{x} \in \mathbb{R}^{n+1}$ such that the following relations hold:
\begin{enumerate}[label=\emph{(\Alph*)}]
\item $D_\pm \vec{x}^j = j \vec{x}^{j-1}, \quad \vec{p}_0^\top \vec{x}^j = a^j, \quad \vec{p}_n^\top \vec{x}^j = b^j, \quad j = 0, \dots, q$,
\item $H = H^\top > 0$,
\item $H D_+ + D_+^\top H = -\vec{p}_0 \vec{p}_0^\top + \vec{p}_n \vec{p}_n^\top + S, \quad S = S^\top \geq 0$,
\item $H D_+ + D_-^\top H = -\vec{p}_0 \vec{p}_0^\top + \vec{p}_n \vec{p}_n^\top$,
\item $\vec{x} = (x_0, \dots, x_s)^\top, \quad x_i \neq x_j \, \forall \, i \neq j$.
\end{enumerate}
\end{definition}

Here and elsewhere the notation $\vec{x}^j$ should be understood as the elementwise exponentiation of $\vec{x}$. The convention $\vec{x}^0 = \vec{1} := (1,\dots,1)^\top$ is used throughout. Definition~\ref{def:SBP} incorporates classical ($S = 0$, $\vec{p}_0 \equiv \vec{e}_0 := (1,0,\dots,0)^\top$, $\vec{p}_n \equiv \vec{e}_n := (0,\dots,0,1)^\top$), generalized ($S = 0$) and upwind SBP methods ($S \neq 0$, $\vec{p}_0 = \vec{e}_0$, $\vec{p}_n = \vec{e}_n$) as special cases.

From (A) it is seen that $D_+$ and $D_-$ approximate derivative operators and that $\vec{p}_0$ and $\vec{p}_n$ interpolate grid functions to the domain boundaries. Further, $H$ defines a quadrature rule \cite{hicken2013summation,linders2018order}. The role of $S$ is somewhat obfuscated in Definition~\ref{def:SBP}, which warrants a comment: Subtracting (D) from (C), transposing and multiplying by $H^{-1}$ shows that $D_- = D_+ - H^{-1}S$. From (A) and (B) it can then be deduced that
\begin{equation} \label{eq:S_conditions}
    S \vec{x}^j = \vec{0}, \quad j = 0, \dots, q.
\end{equation}
Together with positive semi-definiteness, this reveals that $S$ contributes with artificial dissipation to the operator; see e.g. \cite{mattsson2017diagonal,mattsson2004stable}.

SBP operators are frequently used in combination with \emph{simultaneous approximation terms} (SATs) that weakly impose initial, boundary or interface conditions \cite{carpenter1994time}. Suppose that an inflow-outflow problem on $[a,b]$ is described by the the scalar differential equation
\begin{equation} \label{eq:ODE}
    u' = f, \quad u(a) = u_0,
\end{equation}
where $f$ and $u_0$ are given. An SBP-SAT discretization takes the form
\begin{equation} \label{eq:sample_equation}
    D_+ \vec{u} = \vec{f} + \sigma H^{-1} \vec{p}_0 ( u_0 - \vec{p}_0^\top \vec{u} ),
\end{equation}
where the second term on the right-hand side is the SAT. In \eqref{eq:sample_equation}, $\sigma$ is a scalar parameter that is chosen to ensure stability. The choice $\sigma = 1$ is almost exclusively used in practice, hence we focus on this case here.

\begin{remark}
If the flow in \eqref{eq:ODE} is reversed, then $D_-$ is used and the boundary condition is imposed at $b$ using $\vec{p}_n$ and $\sigma = -1$.
\end{remark}

Collecting terms that multiply the solution $u$, \eqref{eq:sample_equation} may be rewritten as
\begin{equation} \label{def:D_tilde}
    \tilde{D}_+ \vec{u} \equiv (D_+ + H^{-1} \vec{p}_0 \vec{p}_0^\top) \vec{u} = \vec{f} + H^{-1} \vec{p}_0 u_0.
\end{equation}
A necessary condition for the existence of a unique solution to \eqref{def:D_tilde} is that the matrix $\tilde{D}_+$ is invertible. SBP operators that have this property can be characterized in terms of \emph{nullspace consistency}, a concept introduced in \cite{svard2019convergence}:

\begin{definition} \label{def:nullspace_consistent}
An SBP operator is said to be nullspace consistent if
$$
\ker{D_+} = \vspan{\vec{1}}.
$$
\end{definition}

For a proof of the following lemma, see \cite[Lemma 2]{linders2020properties}.
\begin{lemma} \label{lemma:nullspace_consistency}
For any SBP operator $D_+$, the matrix $\tilde{D}_+ := (D_+ + H^{-1} \vec{p}_0 \vec{p}_0^\top)$ is invertible if and only if $D_+$ is null-space consistent.
\end{lemma}

Unfortunately, nullspace consistency does not follow from Definition \ref{def:SBP}; counterexamples are given in \cite{ranocha2019some,linders2020properties}. Yet, an even stronger condition is frequently needed:

\begin{definition} \label{def:eigenvalue_property}
An SBP operator is said to have the eigenvalue property if each eigenvalue of $\tilde{D}_+$ has positive real part.
\end{definition}

Many important results in the theory of SBP operators rely on the eigenvalue property. The following list is by no means complete.
\begin{itemize}
    \item The proof that spatial discretizations using (classical) finite difference SBP methods converge with order $q+1$ for hyperbolic problems and order $q+2$ for parabolic problems assumes the eigenvalue property \cite{svard2006order}.
    \item The proof that functional estimates from SBP methods in two or more dimensions are superconvergent requires the eigenvalue property \cite{hicken2011superconvergent}.
    \item SBP methods used for time discretization of partial differential equations need the eigenvalue property to ensure uniqueness of solutions \cite{nordstrom2013summation}. Proofs of convergence for linear and certain nonlinear problems also rely on the property; see \cite{linders2020properties} for an overview.
\end{itemize}

The eigenvalue property is evidently very important and it is therefore of interest to know which SBP operators possess it. Nullspace consistency is of course necessary. This bids the question: Is nullspace consistency sufficient for the eigenvalue property? Unfortunately, this is not the case.

\begin{theorem}
There are nullspace consistent SBP operators that lack the eigenvalue property.
\end{theorem}

\begin{proof}
Consider
$$
D_+ = \frac{1}{5}
\begin{pmatrix}
    -5 & \phantom{-}4 & \phantom{-}2 & \phantom{-}0 & -2 & \phantom{-}1 \\
    -2 & \phantom{-}0 & \phantom{-}1 & \phantom{-}0 & \phantom{-}2 & -1 \\
    -1 & -1 & \phantom{-}0 & \phantom{-}2 & \phantom{-}0 & \phantom{-}0 \\
     \phantom{-}0 & \phantom{-}0 & -2 & \phantom{-}0 & \phantom{-}1 & \phantom{-}1 \\
    \phantom{-}1 & -2 & \phantom{-}0 & -1 & \phantom{-}0 & \phantom{-}2 \\
    -1 & \phantom{-}2 & \phantom{-}0 & -2 & -4 & \phantom{-}5
\end{pmatrix},
$$
defined on the nodes $\vec{x}^\top = (-5,-3,-1,1,3,5)/2$. This is an SBP operator of order $q=1$ with $S=0$, $\vec{p}_0 = \vec{e}_0$, $\vec{p}_n = \vec{e}_n$ and $H = \diag{1/2,1,1,1,1,1/2}$. The matrix $\tilde{D}_+ = D_+ + H^{-1} \vec{p}_0 \vec{p}_0^\top$ is invertible, hence the operator is nullspace consistent by Lemma~\ref{lemma:nullspace_consistency}. However, two of the eigenvalues of $\tilde{D}_+$ and their corresponding eigenvectors are
\begin{equation*} \label{eq:example_eigenvalues}
\begin{aligned}
    \lambda_{\pm} &= \pm \frac{\iu}{\sqrt{5}}, \\
    \vec{w}_{\pm} &= (0,1,-3,3,-1,0)^\top \pm \iu \sqrt{5} (0,1,-1,-1,1,0)^\top.
\end{aligned}
\end{equation*}
The real parts of these eigenvalues are non-positive, hence the SBP operator does not have the eigenvalue property.
\end{proof}


\section{Perturbed SBP operators} \label{sec:Perturbations}

In this section we consider an SBP operator $D_+$ that is nullspace consistent but lack the eigenvalue property, i.e. it satisfies Definition~\ref{def:nullspace_consistent} but not Definition~\ref{def:eigenvalue_property}. It will be shown that the eigenvalue property can be reclaimed by adding an appropriately designed perturbation $S'$ to $D_+$. More precisely, the eigenvalue property is enforced by constructing a new operator $D_+' = D_+ + \frac{1}{2}H^{-1}S'$, which adheres to Definition \ref{def:SBP} so long as $S'$ is symmetric positive semi-definite and satisfies \eqref{eq:S_conditions}. In this case, $S$ is replaced by $S + S'$ in Definition \ref{def:SBP}. The purpose of $S'$ is to push problematic eigenvalues into the right half-plane while leaving the remaining eigenvalues untouched. Furthermore, the matrix $H$ and the vectors $\vec{p}_0$ and $\vec{p}_n$ are left unperturbed. To achieve this, it is first necessary to establish certain properties of the problematic eigenvalues and eigenvectors.

The starting point is the following simple result. For a proof, see e.g. \cite[Lemma 1]{linders2020properties}:

\begin{lemma} \label{lemma:zeros}
Consider the SBP operator $D_+$ and let $(\lambda, \vec{w})$ be an eigenpair of $\tilde{D}_+ = D_+ + H^{-1} \vec{p}_0 \vec{p}_0^\top$, i.e. $\tilde{D}_+ \vec{w} = \lambda \vec{w}$. Then $\operatorname{Re}(\lambda) \geq 0$ with equality if and only if $\vec{p}_0^\top \vec{w} = \vec{p}_n^\top \vec{w} = 0$ and $S \vec{w} = \vec{0}$.
\end{lemma}

Lemma~\ref{lemma:zeros} shows that $\tilde{D}_+$ has an imaginary eigenvalue. By Lemma~\ref{lemma:nullspace_consistency} this eigenvalue is non-zero since $D_+$ is nullspace consistent by assumption. In fact, $\tilde{D}_+$ has an even number of non-zero imaginary eigenvalues since the conjugate $(\overline{\lambda}, \overline{\vec{w}})$ is also an eigenpair of $\tilde{D}_+$.

At this point, two lemmas will be established that are key to the construction of $S'$. Recall that an eigenvalue $\lambda$ of $\tilde{D}_+$ is \emph{normal} if
\begin{enumerate}[label=(\alph*)]
\item every eigenvector of $\tilde{D}_+$ corresponding to $\lambda$ is orthogonal to every eigenvector of $\tilde{D}_+$ corresponding to each eigenvalue different from $\lambda$, and
\item the algebraic and geometric multiplicities of $\lambda$ are equal.
\end{enumerate}
%

\begin{remark}
Herein, the inner product of two vectors $\vec{f}$ and $\vec{g}$ are taken to be $\inner{\vec{f}}{\vec{g}} = \vec{f}^* H \vec{g}$, where the superscript $*$ indicates conjugate transposition. Orthogonality should thus be understood with this inner product in mind. The norm $\| \vec{f} \|_H := \inner{\vec{f}}{\vec{f}}^\frac{1}{2}$ will also be used.
\end{remark}

\begin{lemma} \label{lemma:normal}
Let $(\lambda, \vec{w})$ be an eigenpair of $\tilde{D}_+$ with $\operatorname{Re}(\lambda) = 0$. Then $\lambda$ is a normal eigenvalue with respect to the inner product $\inner{\cdot}{\cdot}$.
\end{lemma}

\begin{proof}
The two properties (a) and (b) must be established. For (a), let $(\mu, \vec{v})$ be any eigenpair of $\tilde{D}_+$ with $\mu \neq \lambda$ and note that $\mu \inner{\vec{w}}{\vec{v}} = \vec{w}^* H \tilde{D}_+ \vec{v}$. Using (C) from Definition \ref{def:SBP} and the definition of $\tilde{D}_+$ in \eqref{def:D_tilde} it follows that
\begin{equation} \label{eq:D_tilde}
    H \tilde{D}_+ = \vec{p}_0 \vec{p}_0^\top + \vec{p}_n \vec{p}_n^\top + S - \tilde{D}_+^\top H.
\end{equation}
Thus,
\begin{align*}
    \mu \inner{\vec{w}}{\vec{v}} &= \underbrace{(\vec{p}_0^\top \vec{w})^* \vec{p}_0^\top \vec{v} + (\vec{p}_n^\top \vec{w})^* \vec{p}_n^\top \vec{v} + (S \vec{w})^* \vec{v}}_{=0 \text{ by Lemma \ref{lemma:zeros}}} - (\tilde{D}_+ \vec{w})^* H \vec{v} \\
    &= - \overline{\lambda} \vec{w}^* H \vec{v} = \lambda \inner{\vec{w}}{\vec{v}}.
\end{align*}
In the final equality, $\overline{\lambda} = -\lambda$ has been used, which holds since $\lambda$ is imaginary. Consequently, $(\mu - \lambda) \inner{\vec{w}}{\vec{v}} = 0$, however since $\mu \neq \lambda$ by assumption, the sought orthogonality follows.

For (b), consider the following problem: Find a symmetric positive definite matrix $X \in \mathbf{R}^{(n+1) \times (n+1)}$ such that $X \tilde{D}_+ + \tilde{D}_+^\top X \geq 0$. From \eqref{eq:D_tilde} it is seen that $X=H$ solves this problem. However, the existence of a solution implies that all elementary divisors of imaginary eigenvalues of $\tilde{D}_+$ are linear \cite[Corollary 2]{carlson1963inertia}, which is equivalent to the stated assertion on the algebraic and geometric multiplicities \cite[Chapter 7]{lancaster1985theory}.
\end{proof}

\begin{lemma} \label{lemma:grid_orthogonal}
Let $(\lambda, \vec{w})$ be an eigenpair of $\tilde{D}_+$ with $\operatorname{Re}(\lambda) = 0$. Then $\inner{\vec{x}^j}{\vec{w}} = 0$ for $j=0,\dots,q$.
\end{lemma}

\begin{proof}
Consider first the case $j=0$, for which we have $\vec{x}^0 = \vec{1}$, and note that $\lambda \inner{\vec{1}}{\vec{w}} = \vec{1}^\top H \tilde{D}_+ \vec{w}$. From Definition \ref{def:SBP} and \eqref{def:D_tilde} it follows that
\begin{equation} \label{eq:D_tilde_D}
    H \tilde{D}_+ = \vec{p}_n \vec{p}_n^\top - D_-^\top H.
\end{equation}
Thus,
\begin{equation*}
    \lambda \inner{\vec{1}}{\vec{w}} = \underbrace{\vec{1}^\top \vec{p}_n (\vec{p}_n^\top \vec{w})}_{=0 \text{ by Lemma \ref{lemma:zeros}}} - \underbrace{(D_- \vec{1})^\top H \vec{w}}_{=0 \text{ by Definition \ref{def:SBP}}} = 0.
\end{equation*}
Since $D_+$ is nullspace consistent by assumption, Lemma~\ref{lemma:nullspace_consistency} ensures that $\lambda \neq 0$. It follows that $\inner{\vec{1}}{\vec{w}} = 0$, hence the claim holds when $j=0$.

Next, suppose that the claim holds for some $j-1 < q$. Then, similarly,
\begin{align*}
    \lambda \inner{\vec{x}^j}{\vec{w}} &= \underbrace{(\vec{x}^j)^\top \vec{p}_n (\vec{p}_n^\top \vec{w})}_{=0 \text{ by Lemma \ref{lemma:zeros}}} - (D_- \vec{x}^j)^\top H \vec{w} \\
    &= -j (\vec{x}^{j-1})^\top H \vec{w} = -j \inner{\vec{x}^{j-1}}{\vec{w}} = 0,
\end{align*}
where the induction hypothesis implies the final equality. From nullspace consistency it follows that $\inner{\vec{x}^j}{\vec{w}}=0$, and the claim holds by induction.
\end{proof}

Using Lemmas \ref{lemma:normal} and \ref{lemma:grid_orthogonal}, the main result of this section can be proven:

\begin{theorem} \label{thm:perturbed_SBP}
For any SBP operator $D_+$ of order $q$ that satisfies Definition~\ref{def:nullspace_consistent} but not Definition~\ref{def:eigenvalue_property}, there is another SBP operator $D'_+ = D_+ + \frac{1}{2} H^{-1} S'$ of order $q$ that satisfies both definitions. Further,  $D_+'$ can be constructed such that $\| D_+' - D_+ \| \leq \varepsilon$ for any $\varepsilon > 0$, where the norm $\| \cdot \|$ is arbitrary.
\end{theorem}

\begin{proof}
Suppose that $\tilde{D}_+$ has precisely $2m$ imaginary eigenvalues, including multiplicity, and denote these $\lambda_1, \overline{\lambda}_1, \dots, \lambda_m, \overline{\lambda}_m$. Select corresponding eigenvectors $\vec{w}_1, \overline{\vec{w}}_1, \dots, \vec{w}_m, \overline{\vec{w}}_m$. By Lemma \ref{lemma:normal} each of these eigenvectors are orthogonal to any eigenvector corresponding to another eigenvalue. Since the algebraic and geometric multiplicities equal for each imaginary eigenvalue, their respective eigenspaces are complete. An orthogonal basis can be found for each space and these basis vectors can be taken as one of the listed eigenvectors. Thus, mutual orthogonality can be ensured among the eigenvectors $\vec{w}_1, \overline{\vec{w}}_1, \dots, \vec{w}_m, \overline{\vec{w}}_m$ even if some of them correspond to the same eigenvalue.

Choose $m$ positive numbers $\epsilon_1, \dots, \epsilon_m$ and construct the matrix
$$
S' = \sum_{k=1}^m \epsilon_k \left( (H \vec{w}_k) (H \vec{w}_k)^* + (H \overline{\vec{w}}_k) (H \overline{\vec{w}}_k)^* \right).
$$
Then $S'$ is real since $(H \overline{\vec{w}}_k) (H \overline{\vec{w}}_k)^* = \overline{(H \vec{w}_k) (H \vec{w}_k)^*}$. Further, $S'$ is symmetric positive semi-definite. 

At this point, construct $D'_+ = D_+ + \frac{1}{2}H^{-1}S'$. Note that
\begin{align*}
S' \vec{x}^j &= \sum_{k=1}^m \epsilon_k \left( (H \vec{w}_k) (H \vec{w}_k)^* + (H \overline{\vec{w}}_k) (H \overline{\vec{w}}_k)^* \right) \vec{x}^j \\
&= \sum_{k=1}^m \epsilon_k \left( (H \vec{w}_k) \inner{\vec{w}_k}{\vec{x}^j} + (H \overline{\vec{w}}_k) \inner{\overline{\vec{w}}_k}{\vec{x}^j} \right) = 0,
\end{align*}
for each $j=0,\dots,q$. Here, Lemma~\ref{lemma:grid_orthogonal} has been used in the final equality. Consequently, $S'$ satisfies \eqref{eq:S_conditions} such that $D'_+$ is an SBP operator of order $q$.

Consider an eigenvalue $\lambda$ of $\tilde{D}_+$. If $\operatorname{Re}(\lambda) \neq 0$, then pick any eigenvector $\vec{v}$ corresponding to $\lambda$ and note that $\tilde{D}'_+ \vec{v} = \tilde{D}_+ \vec{v} = \lambda \vec{v}$ by the orthogonality property in Lemma \ref{lemma:normal}. Thus, the eigenpair $(\lambda, \vec{v})$ is unaltered by $S'$ whenever $\operatorname{Re}(\lambda) \neq 0$. However, if $\operatorname{Re}(\lambda) = 0$ so that $\lambda = \lambda_i$ for some $i$, then pick the appropriate eigenvector $\vec{v} = \vec{w}_i$ and note that
\begin{align*}
    \tilde{D}_+' \vec{w}_i &= \tilde{D}_+ \vec{w}_i + \frac{1}{2} H^{-1} S' \vec{w}_i \\
    &= \lambda_i \vec{w}_i + \frac{\epsilon_i}{2} H^{-1} (H \vec{w}_i) (H \vec{w}_i)^* \vec{w}_i \\
    &= \left( \lambda_i + \frac{\epsilon_i}{2} \| \vec{w}_i \|_H^2 \right) \vec{w}_i,
\end{align*}
where the second equality follows from the orthogonality of the eigenvectors $\vec{w}_1, \overline{\vec{w}}_1, \dots, \vec{w}_m, \overline{\vec{w}}_m$. Thus, $\tilde{D}_+'$ has an eigenpair $(\lambda_i', \vec{w}_i)$, where $\operatorname{Re}(\lambda_i') = \frac{\epsilon_i}{2} \| \vec{w}_i \|_H^2 > 0$. By repeating this procedure for each $\lambda_i$ and $\overline{\lambda}_i$, every eigenvalue of $\tilde{D}'_+$ is accounted for. Consequently, the SBP operator $D_+'$ has the eigenvalue property.

Finally, note that the eigenvectors $\vec{w}_k$ and $\overline{\vec{w}}_k$ may be normalized as desired and that the constants $\epsilon_k$ can be chosen arbitrarily small. Thus, for any norm it is possible to find $S'$ such that for any $\varepsilon$,
$$
\| D_+' - D_+ \| = \frac{1}{2} \| H^{-1} S' \| \leq \varepsilon.
$$
\end{proof}


\section{Pseudospectral methods} \label{sec:Pseudospectral}

An important family of SBP operators consists of the pseudospectral methods. They all satisfy Definition \ref{def:SBP} with order $q=n$; see \cite{carpenter1996spectral, fernandez2014generalized}. These operators are also used in discontinuous Galerkin and flux reconstruction methods \cite{gassner2013skew,ranocha2016summation}. Given $\vec{x}$, the pseudospectral SBP operator $D_+$ is uniquely defined. To see this, note that the accuracy conditions (A) in Definition \ref{def:SBP} imply that the elements $d_{i,j}$ in the $i$th row of $D_+$ satisfy
$$
\underbrace{
\begin{pmatrix}
    1 & \dots & 1 \\
    x_0 & \dots & x_n \\
    \vdots & & \\
    x_0^n & \dots & x_n^n
\end{pmatrix}
}_{V^\top}
\begin{bmatrix}
d_{i,0} \\ d_{i,1} \\ \vdots \\ d_{i,n}
\end{bmatrix}
=
\begin{bmatrix}
0 \\ 1 \\ \vdots \\ n x_i^{n-1}
\end{bmatrix}.
$$
Note that $V$ is a Vandermonde matrix and that the numbers $x_0, \dots, x_n$ are distinct. Thus, $V^\top$ is invertible and $d_{i,0}, \dots, d_{i,n}$ are unique. The same of course holds for each row $i = 0, \dots, n$, hence $D_+$ is uniquely defined.

\begin{remark}
This argument does not imply uniqueness of $H$, $\vec{p}_0$, $\vec{p}_n$ and $S$.
\end{remark}

The goal of this section is to demonstrate that pseudospectral methods defined on arbitrary grids satisfy the eigenvalue property. In \cite{ruggiu2018pseudo} it was shown that this is the case if $H$ defines a quadrature rule that is exact for polynomials of degree $2n-1$ or higher, i.e. if $k \vec{1}^\top H \vec{x}^k = b^{k+1} - a^{k+1}$ for $k=0,\dots,2n-1$. This holds in particular if the grid $\vec{x}$ is chosen to be the Legendre-Gauss, Legendre-Gauss-Radau or Legendre-Gauss-Lobatto nodes (mapped to the interval $[a,b]$) and $H$ is a diagonal matrix containing the corresponding Gaussian quadrature weights. However, for general grids $\vec{x}$, the resulting quadrature rule is exact only for polynomials of degree $n-1$ \cite{fernandez2014generalized}. A different approach is thus necessary to generalize the result from \cite{ruggiu2018pseudo}.

The starting point here is to show that pseudospectral methods on arbitrary grids are nullspace consistent.

\begin{lemma} \label{lemma:pseudospectral_nullspace_consistency}
Every pseudospectral SBP operator is null-space consistent.
\end{lemma}

\begin{proof}
The columns of the Vandermonde matrix $V$ form a basis for $\mathbb{R}^{n+1}$. Thus, if $D_+ \vec{v} = \vec{0}$ and $\vec{v}$ is expanded in terms of this basis as $\vec{v} = \sum_{k=0}^n v_k \vec{x}^k$, then
$$
\vec{0} = D_+ \vec{v} = D_+ \sum_{k=0}^n v_k \vec{x}^k = \sum_{k=0}^n v_k D_+ \vec{x}^k = \sum_{k=1}^n v_k k \vec{x}^{k-1}.
$$
By linear independence of the basis vectors it follows that $v_k = 0$ for $k=1, \dots, n$ and consequently that $\vec{v} = v_0 \vec{1}$. Thus, $\ker{D_+} = \vspan{\vec{1}}$ and $D_+$ is consequently null-space consistent.
\end{proof}

With Lemma~\ref{lemma:pseudospectral_nullspace_consistency} in place, the results from Section~\ref{sec:Perturbations} can be used to establish the eigenvalue property.

\begin{theorem} \label{thm:pseudospectral}
Every pseudospectral SBP operator has the eigenvalue property.
\end{theorem}

\begin{proof}
Suppose that there is a pseudospectral method $D_+$ that does not have the eigenvalue property. Then, by Theorem~\ref{thm:perturbed_SBP}, another operator $D'_+$ of the same order of accuracy can be found that operates on the same grid and that is distinctly different from $D_+$. However, this violates the uniqueness of pseudospectral methods.

Alternatively, suppose that there is a pseudospectral method $D_+$ that does not have the eigenvalue property and denote by $\vec{v}$ one of the eigenvectors of $\tilde{D}_+$ that corresponds to an imaginary eigenvalue. It follows from Lemma~\ref{lemma:grid_orthogonal} that $V^\top H \vec{v} = \vec{0}$. But both $V$ and $H$ are invertible, so $\vec{v} = \vec{0}$ and can therefore not be an eigenvector, which contradicts its definition.
\end{proof}


\section{Summary} \label{sec:Summary}

Three results on the eigenvalues of SBP operators have been proven. Firstly, the eigenvalue property, which is essential in many important accuracy and convergence proofs, does not hold for all nullspace consistent SBP operators. Secondly, this problem can be addressed by carefully constructing an arbitrarily small artificial dissipation term that pushes each problematic eigenvalue into the right half-plane without affecting the other eigenvalues or the order of accuracy of the method. Thirdly, all pseudospectral methods satisfy the eigenvalue property. Thus, SBP operators used within the discontinuous Galerkin and flux reconstruction frameworks are free from the nuisance of problematic eigenvalues.


\bibliographystyle{model1b-num-names}
\bibliography{bibliography.bib}

\begin{thebibliography}{32}
\expandafter\ifx\csname natexlab\endcsname\relax\def\natexlab#1{#1}\fi
\providecommand{\bibinfo}[2]{#2}
\ifx\xfnm\relax \def\xfnm[#1]{\unskip,\space#1}\fi
\bibitem[{Abgrall et~al.(2020)Abgrall, Nordstr{\"o}m, {\"O}ffner and
  Tokareva}]{abgrall2020analysis}
\bibinfo{author}{R.~Abgrall}, \bibinfo{author}{J.~Nordstr{\"o}m},
  \bibinfo{author}{P.~{\"O}ffner}, \bibinfo{author}{S.~Tokareva},
  \bibinfo{title}{Analysis of the {SBP}--{SAT} stabilization for finite element
  methods part {I}: Linear problems}, \bibinfo{journal}{J. Sci. Comput.}
  \bibinfo{volume}{85} (\bibinfo{year}{2020}) \bibinfo{pages}{1--29}.
\bibitem[{Abgrall et~al.(2021)Abgrall, Nordstr{\"o}m, {\"O}ffner and
  Tokareva}]{abgrall2021analysis}
\bibinfo{author}{R.~Abgrall}, \bibinfo{author}{J.~Nordstr{\"o}m},
  \bibinfo{author}{P.~{\"O}ffner}, \bibinfo{author}{S.~Tokareva},
  \bibinfo{title}{Analysis of the {SBP}--{SAT} stabilization for finite element
  methods part {II}: Entropy stability}, \bibinfo{journal}{Commun. Appl. Math.
  Comput.}  (\bibinfo{year}{2021}) \bibinfo{pages}{1--23}.
\bibitem[{Boom and Zingg(2015)}]{boom2015high}
\bibinfo{author}{P.D. Boom}, \bibinfo{author}{D.W. Zingg},
  \bibinfo{title}{High-order implicit time-marching methods based on
  generalized summation-by-parts operators}, \bibinfo{journal}{SIAM J. Sci.
  Comput.} \bibinfo{volume}{37} (\bibinfo{year}{2015})
  \bibinfo{pages}{A2682--A2709}.
\bibitem[{Carlson and Schneider(1963)}]{carlson1963inertia}
\bibinfo{author}{D.~Carlson}, \bibinfo{author}{H.~Schneider},
  \bibinfo{title}{Inertia theorems for matrices: {T}he semidefinite case},
  \bibinfo{journal}{J. Math. Anal. Appl} \bibinfo{volume}{6}
  (\bibinfo{year}{1963}) \bibinfo{pages}{430--446}.
\bibitem[{Carpenter and Gottlieb(1996)}]{carpenter1996spectral}
\bibinfo{author}{M.H. Carpenter}, \bibinfo{author}{D.~Gottlieb},
  \bibinfo{title}{Spectral methods on arbitrary grids}, \bibinfo{journal}{J.
  Comput. Phys.} \bibinfo{volume}{129} (\bibinfo{year}{1996})
  \bibinfo{pages}{74--86}.
\bibitem[{Carpenter et~al.(1994)Carpenter, Gottlieb and
  Abarbanel}]{carpenter1994time}
\bibinfo{author}{M.H. Carpenter}, \bibinfo{author}{D.~Gottlieb},
  \bibinfo{author}{S.~Abarbanel}, \bibinfo{title}{Time-stable boundary
  conditions for finite-difference schemes solving hyperbolic systems:
  {M}ethodology and application to high-order compact schemes},
  \bibinfo{journal}{J. Comput. Phys.} \bibinfo{volume}{111}
  (\bibinfo{year}{1994}) \bibinfo{pages}{220--236}.
\bibitem[{Fern{\'a}ndez et~al.(2014{\natexlab{a}})Fern{\'a}ndez, Boom and
  Zingg}]{fernandez2014generalized}
\bibinfo{author}{D.C.D.R. Fern{\'a}ndez}, \bibinfo{author}{P.D. Boom},
  \bibinfo{author}{D.W. Zingg}, \bibinfo{title}{A generalized framework for
  nodal first derivative {S}ummation-{B}y-{P}arts operators},
  \bibinfo{journal}{J. Comput. Phys.} \bibinfo{volume}{266}
  (\bibinfo{year}{2014}{\natexlab{a}}) \bibinfo{pages}{214--239}.
\bibitem[{Fern{\'a}ndez et~al.(2014{\natexlab{b}})Fern{\'a}ndez, Hicken and
  Zingg}]{fernandez2014review}
\bibinfo{author}{D.C.D.R. Fern{\'a}ndez}, \bibinfo{author}{J.E. Hicken},
  \bibinfo{author}{D.W. Zingg}, \bibinfo{title}{Review of
  {S}ummation-{B}y-{P}arts operators with simultaneous approximation terms for
  the numerical solution of partial differential equations},
  \bibinfo{journal}{Comput. Fluids} \bibinfo{volume}{95}
  (\bibinfo{year}{2014}{\natexlab{b}}) \bibinfo{pages}{171--196}.
\bibitem[{Fisher et~al.(2011)Fisher, Carpenter, Yamaleev and
  Frankel}]{fisher2011boundary}
\bibinfo{author}{T.C. Fisher}, \bibinfo{author}{M.H. Carpenter},
  \bibinfo{author}{N.K. Yamaleev}, \bibinfo{author}{S.H. Frankel},
  \bibinfo{title}{Boundary closures for fourth-order energy stable weighted
  essentially non-oscillatory finite-difference schemes}, \bibinfo{journal}{J.
  Comput. Phys.} \bibinfo{volume}{230} (\bibinfo{year}{2011})
  \bibinfo{pages}{3727--3752}.
\bibitem[{Gassner(2013)}]{gassner2013skew}
\bibinfo{author}{G.J. Gassner}, \bibinfo{title}{A skew-symmetric discontinuous
  {G}alerkin spectral element discretization and its relation to {SBP}--{SAT}
  finite difference methods}, \bibinfo{journal}{SIAM J. Sci. Comput.}
  \bibinfo{volume}{35} (\bibinfo{year}{2013}) \bibinfo{pages}{A1233--A1253}.
\bibitem[{Hicken and Zingg(2011)}]{hicken2011superconvergent}
\bibinfo{author}{J.E. Hicken}, \bibinfo{author}{D.W. Zingg},
  \bibinfo{title}{Superconvergent functional estimates from
  {S}ummation-{B}y-{P}arts finite-difference discretizations},
  \bibinfo{journal}{SIAM J. Sci. Comput.} \bibinfo{volume}{33}
  (\bibinfo{year}{2011}) \bibinfo{pages}{893--922}.
\bibitem[{Hicken and Zingg(2013)}]{hicken2013summation}
\bibinfo{author}{J.E. Hicken}, \bibinfo{author}{D.W. Zingg},
  \bibinfo{title}{Summation-{B}y-{P}arts operators and high-order quadrature},
  \bibinfo{journal}{J. Comput. Anal. Appl.} \bibinfo{volume}{237}
  (\bibinfo{year}{2013}) \bibinfo{pages}{111--125}.
\bibitem[{Kreiss and Scherer(1974)}]{kreiss1974finite}
\bibinfo{author}{H.O. Kreiss}, \bibinfo{author}{G.~Scherer},
  \bibinfo{title}{Finite element and finite difference methods for hyperbolic
  partial differential equations}, in: \bibinfo{booktitle}{Mathematical aspects
  of finite elements in partial differential equations},
  \bibinfo{publisher}{Elsevier}, \bibinfo{year}{1974}, pp.
  \bibinfo{pages}{195--212}.
\bibitem[{Lancaster and Tismenetsky(1985)}]{lancaster1985theory}
\bibinfo{author}{P.~Lancaster}, \bibinfo{author}{M.~Tismenetsky},
  \bibinfo{title}{The theory of matrices: with applications},
  \bibinfo{publisher}{Elsevier}, \bibinfo{year}{1985}.
\bibitem[{Linders et~al.(2016)Linders, Kupiainen, Frankel, Delorme and
  Nordstrom}]{linders2016summation}
\bibinfo{author}{V.~Linders}, \bibinfo{author}{M.~Kupiainen},
  \bibinfo{author}{S.H. Frankel}, \bibinfo{author}{Y.~Delorme},
  \bibinfo{author}{J.~Nordstrom}, \bibinfo{title}{Summation-by-{P}arts
  operators with minimal dispersion error for accurate and efficient flow
  calculations}, in: \bibinfo{booktitle}{54th AIAA Aerospace Sciences Meeting,
  2016}, p. \bibinfo{pages}{1329}.
\bibitem[{Linders et~al.(2017)Linders, Kupiainen and
  Nordstr{\"o}m}]{linders2017summation}
\bibinfo{author}{V.~Linders}, \bibinfo{author}{M.~Kupiainen},
  \bibinfo{author}{J.~Nordstr{\"o}m}, \bibinfo{title}{Summation-by-{P}arts
  operators with minimal dispersion error for coarse grid flow calculations},
  \bibinfo{journal}{J. Comput. Phys.} \bibinfo{volume}{340}
  (\bibinfo{year}{2017}) \bibinfo{pages}{160--176}.
\bibitem[{Linders et~al.(2018)Linders, Lundquist and
  Nordstr\"{o}m}]{linders2018order}
\bibinfo{author}{V.~Linders}, \bibinfo{author}{T.~Lundquist},
  \bibinfo{author}{J.~Nordstr\"{o}m}, \bibinfo{title}{On the order of accuracy
  of finite difference operators on diagonal norm based
  {S}ummation-{B}y-{P}arts form}, \bibinfo{journal}{SIAM J. Numer. Anal.}
  \bibinfo{volume}{56} (\bibinfo{year}{2018}) \bibinfo{pages}{1048--1063}.
\bibitem[{Linders et~al.(2020)Linders, Nordstr{\"o}m and
  Frankel}]{linders2020properties}
\bibinfo{author}{V.~Linders}, \bibinfo{author}{J.~Nordstr{\"o}m},
  \bibinfo{author}{S.H. Frankel}, \bibinfo{title}{Properties of
  {R}unge-{K}utta-{S}ummation-{B}y-{P}arts methods}, \bibinfo{journal}{J.
  Comput. Phys.} \bibinfo{volume}{419} (\bibinfo{year}{2020})
  \bibinfo{pages}{109684}.
\bibitem[{Lundquist and Nordstr{\"o}m(2014)}]{lundquist2014sbp}
\bibinfo{author}{T.~Lundquist}, \bibinfo{author}{J.~Nordstr{\"o}m},
  \bibinfo{title}{The {SBP}-{SAT} technique for initial value problems},
  \bibinfo{journal}{J. Comput. Phys.} \bibinfo{volume}{270}
  (\bibinfo{year}{2014}) \bibinfo{pages}{86--104}.
\bibitem[{Mattsson(2017)}]{mattsson2017diagonal}
\bibinfo{author}{K.~Mattsson}, \bibinfo{title}{Diagonal-norm upwind {SBP}
  operators}, \bibinfo{journal}{J. Comput. Phys.} \bibinfo{volume}{335}
  (\bibinfo{year}{2017}) \bibinfo{pages}{283--310}.
\bibitem[{Mattsson et~al.(2004)Mattsson, Sv{\"a}rd and
  Nordstr{\"o}m}]{mattsson2004stable}
\bibinfo{author}{K.~Mattsson}, \bibinfo{author}{M.~Sv{\"a}rd},
  \bibinfo{author}{J.~Nordstr{\"o}m}, \bibinfo{title}{Stable and accurate
  artificial dissipation}, \bibinfo{journal}{J. Sci. Comput.}
  \bibinfo{volume}{21} (\bibinfo{year}{2004}) \bibinfo{pages}{57--79}.
\bibitem[{Nordstr{\"o}m and Bj{\"o}rck(2001)}]{nordstrom2001finite}
\bibinfo{author}{J.~Nordstr{\"o}m}, \bibinfo{author}{M.~Bj{\"o}rck},
  \bibinfo{title}{Finite volume approximations and strict stability for
  hyperbolic problems}, \bibinfo{journal}{Appl. Numer. Math.}
  \bibinfo{volume}{38} (\bibinfo{year}{2001}) \bibinfo{pages}{237--255}.
\bibitem[{Nordstr{\"o}m and Linders(2018)}]{nordstrom2018well}
\bibinfo{author}{J.~Nordstr{\"o}m}, \bibinfo{author}{V.~Linders},
  \bibinfo{title}{Well-posed and stable transmission problems},
  \bibinfo{journal}{J. Comput. Phys.} \bibinfo{volume}{364}
  (\bibinfo{year}{2018}) \bibinfo{pages}{95--110}.
\bibitem[{Nordstr{\"o}m and Lundquist(2013)}]{nordstrom2013summation}
\bibinfo{author}{J.~Nordstr{\"o}m}, \bibinfo{author}{T.~Lundquist},
  \bibinfo{title}{Summation-{B}y-{P}arts in time}, \bibinfo{journal}{J. Comput.
  Phys.} \bibinfo{volume}{251} (\bibinfo{year}{2013})
  \bibinfo{pages}{487--499}.
\bibitem[{Ranocha(2019)}]{ranocha2019some}
\bibinfo{author}{H.~Ranocha}, \bibinfo{title}{Some notes on
  {S}ummation-{B}y-{P}arts time integration methods}, \bibinfo{journal}{Results
  Appl. Math.} \bibinfo{volume}{1} (\bibinfo{year}{2019})
  \bibinfo{pages}{100004}.
\bibitem[{Ranocha et~al.(2016)Ranocha, {\"O}ffner and
  Sonar}]{ranocha2016summation}
\bibinfo{author}{H.~Ranocha}, \bibinfo{author}{P.~{\"O}ffner},
  \bibinfo{author}{T.~Sonar}, \bibinfo{title}{Summation-{B}y-{P}arts operators
  for correction procedure via reconstruction}, \bibinfo{journal}{J. Comput.
  Phys.} \bibinfo{volume}{311} (\bibinfo{year}{2016})
  \bibinfo{pages}{299--328}.
\bibitem[{Ruggiu and Nordstr{\"o}m(2018)}]{ruggiu2018pseudo}
\bibinfo{author}{A.A. Ruggiu}, \bibinfo{author}{J.~Nordstr{\"o}m},
  \bibinfo{title}{On pseudo-spectral time discretizations in
  {S}ummation-{B}y-{P}arts form}, \bibinfo{journal}{J. Comput. Phys.}
  \bibinfo{volume}{360} (\bibinfo{year}{2018}) \bibinfo{pages}{192--201}.
\bibitem[{Strand(1994)}]{strand1994summation}
\bibinfo{author}{B.~Strand}, \bibinfo{title}{Summation-{B}y-{P}arts for finite
  difference approximations for $d/dx$}, \bibinfo{journal}{J. Comput. Phys.}
  \bibinfo{volume}{110} (\bibinfo{year}{1994}) \bibinfo{pages}{47--67}.
\bibitem[{Sv{\"a}rd and Nordstr{\"o}m(2006)}]{svard2006order}
\bibinfo{author}{M.~Sv{\"a}rd}, \bibinfo{author}{J.~Nordstr{\"o}m},
  \bibinfo{title}{On the order of accuracy for difference approximations of
  initial-boundary value problems}, \bibinfo{journal}{J. Comput. Phys.}
  \bibinfo{volume}{218} (\bibinfo{year}{2006}) \bibinfo{pages}{333--352}.
\bibitem[{Sv{\"a}rd and Nordstr{\"o}m(2014)}]{svard2014review}
\bibinfo{author}{M.~Sv{\"a}rd}, \bibinfo{author}{J.~Nordstr{\"o}m},
  \bibinfo{title}{Review of {S}ummation-{B}y-{P}arts schemes for
  initial-boundary-value problems}, \bibinfo{journal}{J. Comput. Phys.}
  \bibinfo{volume}{268} (\bibinfo{year}{2014}) \bibinfo{pages}{17--38}.
\bibitem[{Sv{\"a}rd and Nordstr{\"o}m(2019)}]{svard2019convergence}
\bibinfo{author}{M.~Sv{\"a}rd}, \bibinfo{author}{J.~Nordstr{\"o}m},
  \bibinfo{title}{On the convergence rates of energy-stable finite-difference
  schemes}, \bibinfo{journal}{J. Comput. Phys.} \bibinfo{volume}{397}
  (\bibinfo{year}{2019}) \bibinfo{pages}{108819}.
\bibitem[{Yamaleev and Carpenter(2009)}]{yamaleev2009systematic}
\bibinfo{author}{N.K. Yamaleev}, \bibinfo{author}{M.H. Carpenter},
  \bibinfo{title}{A systematic methodology for constructing high-order energy
  stable {WENO} schemes}, \bibinfo{journal}{J. Comput. Phys.}
  \bibinfo{volume}{228} (\bibinfo{year}{2009}) \bibinfo{pages}{4248--4272}.

\end{thebibliography}


\end{document}